\documentclass[12pt]{amsart}

\usepackage[pagebackref=true,colorlinks,linkcolor=blue,citecolor=magenta]{hyperref}
\usepackage{amsmath,amssymb,latexsym} 
\usepackage{graphicx}
\usepackage{fullpage}
\usepackage[square, numbers]{natbib}   
\usepackage{comment}
\usepackage{etoolbox}

\newtheorem{theorem}{Theorem} 
\newtheorem{corollary}{Corollary}
\newtheorem{proposition}{Proposition}
\newtheorem{claim}{Claim}
\newtheorem{lemma}{Lemma}

\theoremstyle{definition}

\theoremstyle{remark}

\newtheorem*{proof-claim}{Proof}

\newtheorem{prealphthm}{{\bf Theorem}}

\newtheorem{prealphlemma}{{\bf Lemma}}

\newenvironment{changemargin}[2]{\begin{list}{}{%
\setlength{\topsep}{0pt}%
\setlength{\leftmargin}{0pt}%
\setlength{\rightmargin}{0pt}%
\setlength{\listparindent}{\parindent}%
\setlength{\itemindent}{\parindent}%
\setlength{\parsep}{0pt plus 1pt}%
\addtolength{\leftmargin}{#1}%
\addtolength{\rightmargin}{#2}%
}\item }{\end{list}}

\AtBeginEnvironment{proof-claim}{\vspace{-0.5cm}\footnotesize\begin{changemargin}{0.5cm}{0.5cm}}
\AtEndEnvironment{proof-claim}{\end{changemargin}}

\def\K{\mathsf{K}}
\def\L{\mathsf{L}}
\def\R{\mathbb{R}}
\def\Z{\mathbb{Z}}
\def\T{\mathsf{T}}

\def\S{\mathcal{S}}

\def\ee{\boldsymbol{e}}
\def\ff{\boldsymbol{f}}

\def\xx{\boldsymbol{x}}
\def\yy{\boldsymbol{y}}
\def\zz{\boldsymbol{z}}

\def\zero{\boldsymbol{0}}

\def\ds{\displaystyle}
\def\id{\operatorname{id}}
\def\KG{\operatorname{KG}}
\def\cd{\operatorname{cd}}

\def\alt{\operatorname{alt}}
\def\H{\mathcal{H}}
\def\id{\operatorname{id}}

\def\sd{\operatorname{sd}}
\def\mod{\operatorname{mod}}
\def\ind{\operatorname{ind}}
\def\coind{\operatorname{coind}}
\def\Xind{\operatorname{Xind}}
\def\Hom{\operatorname{Hom}}
\def\NP{\operatorname{NP}}

\title{Strengthening topological colorful results\\ for graphs \footnote{\copyright~2017~ This manuscript version is made available under the CC-BY-NC-ND 4.0 license \url{http://creativecommons.org/licenses/by-nc-nd/4.0}}}

\author{Meysam Alishahi}
\address{M. Alishahi,
School of Mathematical Sciences,
Shahrood University of Technology, Shahrood, Iran}
\email{meysam\_alishahi@shahroodut.ac.ir}
\author{Hossein Hajiabolhassan}
\address{H. Hajiabolhassan,
Department of Mathematical Sciences,
Shahid Beheshti University, P.O. Box 169411, Tehran, Iran \newline
School of Mathematics, Institute for Research in Fundamental Sciences (IPM), P.O. Box 19395-5746, Tehran, Iran}
\email{hhaji@sbu.ac.ir}
\author{Fr\'ed\'eric Meunier}
\address{F. Meunier, Universit\'e Paris Est, CERMICS, 77455 Marne-la-Vall\'ee CEDEX, France}
\email{frederic.meunier@enpc.fr}


\begin{document}

\maketitle

\begin{abstract}
Various results ensure the existence of large complete and colorful bipartite graphs in properly colored graphs when some condition related to a topological lower bound on the chromatic number is satisfied. We generalize three theorems of this kind, respectively due to Simonyi and Tardos ({\em Combinatorica}, 2006), Simonyi, Tardif, and Zsb\'an ({\em The Electronic Journal of Combinatorics}, 2013), and Chen ({\em Journal of Combinatorial Theory, Series A}, 2011). As a consequence of the generalization of Chen's theorem, we get new families of graphs whose chromatic number equals their circular chromatic number and that satisfy Hedetniemi's conjecture for the circular chromatic number.
\end{abstract}

\section{Introduction}\label{intro}

The famous proof of the Kneser conjecture by Lov\'asz opened the door to the use of topological tools in combinatorics, especially for finding topological obstructions for a graph $G$ to have a small chromatic number. These obstructions actually often ensure the existence of large complete bipartite subgraphs with many colors in any proper coloring of $G$. Finding new conditions for the existence of such bipartite graphs has in particular become an active stream of research relying on the use of topological tools. 
Apart of being elegant, these results have also turned out to be useful, e.g., for providing lower bounds for various types of chromatic numbers. The {\em zig-zag theorem}, the {\em colorful $K_{\ell,m}$ theorem}, and the {\em alternative Kneser coloring lemma} are among the most prominent results of this stream and our main objective is to provide generalizations of them. \\

The topological obstructions often take the form of a lower bound provided by a topological invariant attached to the graph $G$. One of these topological invariants is the ``coindex of the box-complex'' $\coind(B_0(G))$  (the exact definition of this topological invariant, as well as others, is recalled later, in Section~\ref{subsubsec:graphcomplex}). It has been shown that $$\chi(G)\geq\coind(B_0(G))+1.$$ The zig-zag theorem, due to Simonyi and Tardos~\cite{SiTa06}, is a generalization of this inequality and ensures that any proper coloring of $G$ contains a heterochromatic complete bipartite subgraph $K_{\lceil\frac t 2\rceil,\lfloor \frac t 2\rfloor}$, where $t=\coind(B_0(G))+1$. {\em Heterochromatic} means that the vertices get pairwise distinct colors. In Section~\ref{sec:zigzag}, we provide the exact statement of the zig-zag theorem and prove a   result (Theorem~\ref{thm:path_colorful}) that extends it roughly as follows. Two complete bipartite subgraphs are {\em adjacent} if one is obtained from the other by removing a vertex and adding a vertex, possibly the same if the vertex is added to the other side of the bipartite graph. We prove that in any proper coloring of $G$, there is a sequence of adjacent ``almost'' heterochromatic complete bipartite subgraphs, which starts with a heterochromatic complete bipartite graph as in the original statement of the zig-zag theorem and ends at the same bipartite graph, with the two sides interchanged. \\

While the zig-zag theorem holds for any properly colored graph, the colorful $K_{\ell,m}$ theorem, also due to Simonyi and Tardos~\cite{SiTa07}, requires $\chi(G)$ to be equal to $\coind(B_0(G))+1$. Several families of graphs satisfy this equality, e.g., Kneser, Schrijver, and Borsuk graphs, and the Mycielski construction keeps the equality when applied to a graph satisfying it, see the aforementioned paper for details. Assume that such a graph $G$ is properly colored with a minimum number of colors and consider a bipartition $I,J$ of the color set. The colorful $K_{\ell,m}$ theorem ensures then that there is a complete bipartite subgraph with the colors in $I$ on one side, and the colors in $J$ on the other side. The name of the theorem reminds that, contrary to the zig-zag theorem, one can choose the sizes of the sides of the bipartite graph. {\em Colorful}, here, and in the remaining of the paper, means that all colors appear in the subgraph.

In Section~\ref{sec:colorful}, we show that the colorful $K_{\ell,m}$ theorem remains true when $\coind(B_0(G))+1$ is replaced by the ``cross-index of the Hom complex'', a combinatorial invariant providing a better bound on the chromatic number. This answers an open question by Simonyi, Tardif, and Zsb\'an, see p.6 of \cite{SiTaZs13}.\\

A hypergraph $\H$ is a pair $(V,E)$ where $V$ -- the {\em vertex set} -- is a nonempty finite set and where $E$ -- the {\em edge set} -- is a finite collection of subsets of $V$. A hypergraph is {\em nonempty} if it has at least one edge. A {\em singleton} is an edge with a single vertex. Given a hypergraph $\H$, the {\em general Kneser graph} $\KG(\H)$ is a graph whose vertices are the edges of $\H$ and in which two vertices are connected if the corresponding edges are disjoint. The {\em $2$-colorability defect} of $\H$, denoted by $\cd_2(\H)$, is the minimal number of vertices to remove from $\H$ such that the remaining partial hypergraph is $2$-colorable:
$$\cd_2(\H)=\min\left\{|U|: \left(V(\H)\setminus U,\{e\in E(\H):\;e\cap U=\varnothing\}\right)\mbox{ is $2$-colorable}\right\}.$$
A hypergraph is {\it $2$-colorable} whenever it is possible to color its 
vertices using $2$ colors in such a way that no edge is monochromatic. 
Dol'nikov~\cite{Do88} proved that the colorability defect of $\H$ is a lower bound of the chromatic number of $\KG(\H)$:
\begin{equation}\label{eq:dolnikov}
\chi(\KG(\H))\geq\cd_2(\H).
\end{equation}
Dol'nikov's inequality actually provides a lower bound for the chromatic number of every graph since for every simple graph $G$, there exists a hypergraph $\H$, a {\em Kneser representation of $G$}, such that $\KG(\H)$ and $G$ are isomorphic. Denote by $K_{q,q}^*$ the bipartite graph obtained from the complete bipartite graph $K_{q,q}$ by removing a perfect matching. 

In Section~\ref{sec:alternative}, we prove the following result, which shows that when Dolnikov's inequality is an equality, any proper coloring with a minimal number of colors presents a special pattern.
\begin{corollary}\label{cor:cd2}
Let $\H$ be a nonempty hypergraph such that $\chi(\KG(\H))=\cd_2(\H)=t$ and with no singletons. Any proper coloring of $\KG(\H)$ with $t$ colors contains a colorful copy of $K_{t,t}^*$ with all colors appearing on each side.
\end{corollary}
We actually prove a more general result (Theorem~\ref{thm:main}) ensuring the existence of such a $K_{t,t}^*$ in the categorical product of an arbitrary number of general Kneser graphs, where the ``categorical product'' is a way to multiply graphs (it is the product involved in the famous Hedetniemi conjecture). 

The alternative Kneser coloring lemma, found by Chen in 2011~\cite{Ch11}, is the special case of Corollary~\ref{cor:cd2} when $\H$ is the hypergraph $\left([n],{{[n]}\choose k}\right)$ (the complete $k$-uniform hypergraph with $[n]$ as the vertex set) with $k\geq 2$. For this special case, the graph $\KG(\H)$ is  a ``usual'' Kneser graph, denoted by $\KG(n,k)$, which plays an important role in combinatorics and in graph theory, see for instance~\cite{Ha10,SiTa06,SiTa07,St76,VaVe05}. When $k\geq 2$, the usual Kneser graph satisfies the condition of Corollary~\ref{cor:cd2} with $t=n-2k+2$ (the fact that the chromatic number is $n-2k+2$ is the Lov\'asz theorem~\cite{Lo79}, originally conjectured by Kneser~\cite{kneser1955}). Note that Corollary~\ref{cor:cd2} is a strict improvement on the colorful $K_{\ell,m}$ theorem for general Kneser graphs satisfying its condition.
The existence of such colorful bipartite subgraphs has immediate consequences for the circular chromatic number, an important notion in graph coloring theory, see Section~\ref{subsec:circ} for the definition of the circular chromatic number and for these consequences. \\

Motivated by Corollary~\ref{cor:cd2}, we end the paper in Section~\ref{sec:cd} with a discussion on hypergraphs $\H$ such that $\chi(\KG(\H))=\cd_2(\H)$. It turns out that this question seems to be difficult. We can describe procedures to build such hypergraphs, but a general characterization is still missing. We do~not even know what is the complexity of deciding whether $\H$ satisfies $\chi(\KG(\H))=\cd_2(\H)$.\\

Before stating and proving these results, Section~\ref{sec:basic} presents the main topological tools used in the paper. It contains no new results but we assume that the reader has basic knowledge in topological combinatorics. More details on that topic can be found for instance in the books by Matou{\v{s}}ek~\cite{Matoubook} or De Longueville~\cite{DeLongbook}.

\section{Topological tools}\label{sec:basic}

\subsection{Alternation number}\label{subsec:alt_cd}

%

An improvement on the $2$-colorability defect is the alternation number, which provides a better lower bound on the chromatic number of general Kneser graphs. It has been defined by the first and second authors~\cite{AlHa15}. Let $\xx=(x_1,\ldots,x_n)\in\{+,-,0\}^n$. An {\em alternating subsequence} of $\xx$ is a sequence $x_{i_1},\ldots,x_{i_{\ell}}$ of nonzero terms of $\xx$, with $i_1<\cdots<i_{\ell}$, such that any two consecutive terms of this sequence are different. We denote by $\alt(\xx)$ the length of a longest alternating subsequence of $\xx$. 
In the case $\xx=\zero$, we define $\alt(\xx)$ to be $0$. For a nonempty hypergraph $\H$ and a bijection $\sigma:[n]\rightarrow V(\H)$, we define $$\alt_{\sigma}(\H)=\max_{\xx\in\{+,-,0\}^n}\big\{\alt(\xx):\;\mbox{none of $\sigma(\xx^+)$ and $\sigma(\xx^-)$ contains an edge of }\H\big\},$$ where $$\xx^+=\{i:\;x_i=+\}\qquad\mbox{and}\qquad\xx^-=\{i:\;x_i=-\}.$$
 The {\em alternation number} of $\H$ is the quantity $\alt(\H)=\min_{\sigma}\alt_{\sigma}(\H)$, where the minimum is taken over all bijections $\sigma:[n]\rightarrow V(\H)$. The following inequality is proved in the same paper
\begin{equation}\label{eq:alter}
\chi(\KG(\H))\geq n-\alt(\H).
\end{equation}
Note that this implies $\chi(\KG(\H))\geq n-\alt_\sigma(\H)$ for any bijection 
$\sigma: [n]\rightarrow V(\H)$ and thus it implies Dol'nikov's inequality~\eqref{eq:dolnikov} since $n-\alt_{\sigma}(\H)\geq\cd_2(\H)$ holds for every bijection $\sigma:[n]\rightarrow V(\H)$.
The first two authors have applied inequality~\eqref{eq:alter} on various families of graphs~\cite{2014arXiv1401.0138A,AlHa14,2014arXiv1407.8035A,AlHa15_arxiv,2015arXiv150708456A}.

\subsection{Box complex, Hom complex, and indices}

\subsubsection{Some topological notions}

A {\em $\Z_2$-space} $X$ is a topological space with a homeomorphism $\nu:X\rightarrow X$, the {\em $\Z_2$-action on $X$}, such that $\nu\circ\nu=\id_X$. If $\nu$ has no fixed points, then $\nu$ is {\em free}. In that case, $X$ is also called free. A {\em $\Z_2$-map} $f:X\rightarrow Y$ between two $\Z_2$-spaces is a continuous map such that $f\circ\nu=\omega\circ f$, where $\nu$ and $\omega$ are the $\Z_2$-actions on $X$ and $Y$ respectively. Given two $\Z_2$-spaces $X$ and $Y$, we write $X \stackrel{\Z_2}{\longrightarrow} Y$ if there exists a $\Z_2$-map $X\rightarrow Y$.

The {\em $\Z_2$-index} and {\em $\Z_2$-coindex} of a $\Z_2$-space $X$ are defined as follows:
$$\ind(X)=\min\{d\geq 0:\; X \stackrel{\Z_2}{\longrightarrow}\S^d\}\qquad\mbox{and}\qquad\coind(X)=\max\{d\geq 0:\; \S^d \stackrel{\Z_2}{\longrightarrow} X\}, $$
where the $\Z_2$-action on the $d$-dimensional sphere $\S^d$ is the (free) antipodal map $-:\xx\rightarrow -\xx$. The celebrated Borsuk-Ulam theorem states that $\ind(\S^d)=\coind(\S^d)=d$. It also implies that $\ind(X)\geq\coind(X)$. \\

These definitions extend naturally to simplicial complexes. A {\em simplicial $\Z_2$-complex} $\K$ is a simplicial complex with a simplicial map $\nu:\K\rightarrow\K$, the {\em $\Z_2$-action on $\K$}, such that $\nu\circ\nu=\id_{\K}$. The underlying space of $\K$, denoted by $\|\K\|$, becomes a $\Z_2$-space if we take the affine extension of $\nu$ as a $\Z_2$-action on $\|\K\|$. If $\nu$ has no fixed simplices, then $\nu$ is {\em free} and $\K$ is a {\em free simplicial $\Z_2$-complex}. If we take the affine extension of $\nu$ as a $\Z_2$-action on $\|\K\|$, then $\K$ is free if and only if $\|\K\|$ is free. A {\em simplicial $\Z_2$-map} $\lambda:\K\rightarrow\L$ is a simplicial map between two simplicial $\Z_2$-complexes such that $\lambda\circ\nu=\omega\circ\lambda$, where $\nu$ and $\omega$ are the $\Z_2$-actions of $\K$ and $\L$ respectively. Given two simplicial $\Z_2$-complexes $\K$ and $\L$, we write $\K \stackrel{\Z_2}{\longrightarrow} \L$ if there exists a simplicial $\Z_2$-map $\K\rightarrow\L$. Note that if $\K \stackrel{\Z_2}{\longrightarrow} \L$, then $\|\K\| \stackrel{\Z_2}{\longrightarrow}\|\L\|$, but the converse is in general not true.

The $\Z_2$-index and $\Z_2$-coindex of a simplicial $\Z_2$-complex $\K$ are defined as $\ind(\K)=\ind(\|\K\|)$ and $\coind(\K)=\coind(\|\K\|)$, where again the $\Z_2$-action on $\|\K\|$ is the affine extension of the $\Z_2$-action on $\K$. \\

A {\em free $\Z_2$-poset} $P$ is an ordered set with a fixed-point free automorphism $-:P\rightarrow P$ of order $2$ (being understood that this automorphism is order-preserving). An {\em order-preserving $\Z_2$-map} between two free $\Z_2$-posets
$P$ and $Q$ is an order-preserving map $\phi:P\rightarrow Q$ such that $\phi(-x)=-\phi(x)$ for every $x\in P$. If there exists such a map, we write $P \stackrel{\Z_2}{\longrightarrow} Q$. For a nonnegative integer $n$, let $Q_n$ be the free $\Z_2$-poset with elements $\{\pm 1,\ldots,\pm (n+1)\}$ such that for any two members $x$ and $y$ in $Q_n$, we have $x<_{Q_n}y$ if $|x|<|y|$. For a free $\Z_2$-poset $P$, the {\em cross-index} of $P$, denoted by $\Xind(P)$, is the minimum $n$ such that there is a $\Z_2$-map from $P$ to $Q_n$. If $P$ has no elements, we define $\Xind(P)$ to be $-1$.

Given a poset $P$, its {\em order complex}, denoted by $\Delta P$, is the simplicial complex whose simplices are the chains of $P$. If $P$ is a free $\mathbb{Z}_2$-poset, then $\Delta P$ is a free simplicial $\Z_2$-complex, i.e., no chains are fixed under the automorphism $-$. Any order-preserving $\Z_2$-map between two free $\Z_2$-posets $P$ and $Q$ lifts naturally to a $\Z_2$-simplicial map between $\Delta P$ and $\Delta Q$. Since  there is a $\Z_2$-map from $\|\Delta Q_n\|$ to $\S^n$, we have $\Xind(P)\geq\ind(\Delta P)$ for any free $\Z_2$-poset $P$. In the sequel, we write $\ind(P)$ for $\ind(\Delta P)$.

\subsubsection{Complexes for graphs}\label{subsubsec:graphcomplex}

Let $G=(V,E)$ be a graph. The {\em box complex of $G$}, denoted by $B_0(G)$, has $V \uplus V = V\times[2]$ as the vertex set and
$$\{A\uplus B:\ A,B\subseteq V,\ A\cap B=\varnothing,\ G[A,B] \mbox{ is a complete bipartite graph} \}$$ as the simplex set, where $G[A,B]$ is the bipartite graph whose sides are $A$ and $B$ and whose edges are those of $G$ having one endpoint in $A$ and the other in $B$. 
The $\Z_2$-action $\nu\colon V\uplus V\rightarrow V \uplus V$ on $B_0(G)$ is given by interchanging the two copies of $V$;
that is, $\nu(v,1)=(v,2)$ and $\nu(v, 2)=(v, 1)$, for any $v \in V $.
Clearly, this $\Z_2$-action makes $B_0(G)$ a free simplicial $\Z_2$-complex. 

There is also another box complex, denoted by $B(G)$, which differs from $B_0(G)$ in that if one of $A$ or $B$ is empty, the vertices in the other one must still have a common neighbor. Since we are not using this complex any further, we are not giving more details on it.

The {\em Hom complex} $\Hom(K_2, G)$ of a graph $G$ is a free $\Z_2$-poset consisting of all pairs $(A,B)$ such that $A$ and $B$ are nonempty disjoint subsets of $V$ and such that $G[A,B]$ is a complete bipartite graph. The partial order of this poset is the inclusion $\subseteq$ extended to pairs of subsets of $V$:
$$(A,B)\subseteq (A',B')\quad\mbox{if $A\subseteq A'$ and $B\subseteq B'$}.$$
The fixed-point free automorphism $-$ on $\Hom(K_2, G)$ is defined by $-(A,B)=(B,A)$. The definitions of the box complex $B_0(G)$ and of the Hom complex have some similarity. To understand their relationships further, see~\cite{SiTa06} for instance.

These complexes can be used to provide lower bounds on the chromatic number of $G$, see \cite{AlHa14,AlHa15,MaZi04,SiTaZs13}.
They are related by the following chain of inequalities,
\begin{equation}\label{lbchrom}
\begin{array}{lll}
\chi(G) & \geq & \Xind(\Hom(K_2, G))+2   \geq \ind(\Hom(K_2,G))+2\geq \ind(B_0(G))+1\\
& \geq & \coind(B_0(G))+1 \geq |V(\H)|-\alt(\H) \geq\cd_2(\H),
\end{array}
\end{equation}
where $\H$ is any Kneser representation of $G$.

\subsection{Fan's lemma and its variations}\label{subsec:kyfan}

Fan's lemma, which plays an important role in topology and in combinatorics, is the following theorem.

\begin{theorem}[Fan lemma~\cite{Fa56}]\label{thm:kyfan}
Consider a centrally-symmetric triangulation $\T$ of $\S^d$ and let $\lambda:V(\T)\rightarrow\{\pm 1,\ldots,\pm m\}$ be a labeling satisfying the following properties:
\begin{itemize}
\item it is antipodal: $\lambda(-v)=-\lambda(v)$ for every $v\in V(\T)$
\item it has no complementary edges: $\lambda(u)+\lambda(v)\neq 0$ for every pair  $uv$ of adjacent vertices of $\T$.
\end{itemize}
Then there is an odd number of $d$-dimensional simplices $\sigma$ of $\T$ such that $\lambda(V(\sigma))$ is of the form $$\{-j_1,+j_2,\ldots,(-1)^{d+1}j_{d+1}\}$$ for some $1\leq j_1<\cdots<j_{d+1}\leq m$. In particular, $m\geq d+1$.
\end{theorem}
Such simplices $\sigma$ are {\em negative-alternating simplices}.

This theorem is especially useful to combinatorics in the special case when $\T$ is the first barycentric subdivision of the boundary of the $n$-dimensional cross-polytope. It takes then a purely combinatorial aspect. Before stating this special case, we define a partial order that plays a role in the remaining of the paper. As used in oriented matroid theory, we define $\preceq$ to be the following partial order on $\{+,-,0\}$: $$0\preceq +,\quad 0\preceq -,\quad+\mbox{ and }-\mbox{ are not comparable.}$$ We extend it for sign vectors by simply taking the product order: for $\xx,\yy\in\{+,-,0\}^n$, we have $\xx\preceq\yy$ if the following implication holds for every $i\in[n]$
$$x_i\neq 0\Longrightarrow x_i=y_i.$$


\begin{lemma}[Octahedral Fan lemma]\label{lem:kyfan_oct}
Let $m$ and $n$ be positive integers and
$$\lambda:\{+,-,0\}^n\setminus \{\zero\} \longrightarrow\{\pm 1,\ldots,\pm m\}$$
be a map satisfying the following properties:
\begin{itemize}
\item $\lambda(-\xx)=-\lambda(\xx)$ for all $\xx$,
\item $\lambda(\xx)+\lambda(\yy)\neq 0$ when $\xx\preceq\yy$.
\end{itemize}
Then there is an odd number of chains
$\xx_1\preceq \cdots \preceq \xx_n$ such that
$\ds\left\{\lambda(\xx_1),\ldots,\lambda(\xx_n)\right\}$ is of the form $$\{-j_1,+j_2,\ldots,(-1)^nj_n\}$$ for some $1\leq j_1<\cdots<j_n\leq m$. In
particular, $m \geq n$.
\end{lemma}

The following result is proved with the help of Lemma~\ref{lem:kyfan_oct}. It is implicitly used 
in the proof of the alternative Kneser coloring lemma in \cite{Ch11}. Here, we state it as 
a lemma, with a proof for the sake of completeness.

\begin{lemma}\label{lem:chen}
Consider an order-preserving $\Z_2$-map $\lambda:\left(\{+,-,0\}^n\setminus\{\zero\},\preceq\right)\rightarrow Q_{n-1}$.
Suppose moreover that there is a $\gamma\in[n]$ such that when $\xx\prec\yy$, at most one of $|\lambda(\xx)|$ and $|\lambda(\yy)|$ is equal to $\gamma$.
Then there are two chains $$\xx_{1}\preceq\cdots\preceq\xx_{n}\quad\mbox{and}\quad\yy_{1}\preceq\cdots\preceq\yy_{n}$$ such that
$$\lambda(\xx_{i})=(-1)^ii\quad \mbox{for all $i$}\qquad\mbox{and }\qquad\lambda(\yy_{i})=(-1)^ii \quad\mbox{for $i\neq\gamma$}$$
and such that $\xx_{\gamma}=-\yy_{\gamma}$.
\end{lemma}

\begin{proof}
Let $\Gamma$ be the set of all vectors $\xx\in \{+,-,0\}^n\setminus\{\zero\}$ such that $|\lambda(\xx)|=\gamma$. 
For $\xx\in \Gamma$, define $\alpha(\xx,\lambda)$ to be the number of chains $\xx_{1}\preceq\cdots\preceq\xx_{n}$ that contain $\xx$ and such that $\lambda(\xx_{i})=(-1)^ii\quad \mbox{for all $i$}$. Such a chain contains exactly one $\xx$ of $\Gamma$.
Thus, using Lemma~\ref{lem:kyfan_oct}, we get that $\ds\sum_{\xx\in \Gamma}\alpha(\xx,\lambda)$ is an odd integer. It implies that there is at least one $\zz\in \Gamma$ such that
$\alpha(\zz,\lambda)$ is odd.
Now, define
$\lambda':(\{+,-,0\}^n\setminus\{\zero\},\preceq)\rightarrow Q_{n-1}$ such that
$$\lambda'(\xx)=\left\{
\begin{array}{ll}
-\lambda(\xx) & \mbox{if } \xx\in\{\zz,-\zz\}\\
\lambda(\xx) & \mbox{otherwise.}
\end{array}\right.$$
Because of the property enjoyed by $\gamma$, when $\xx$ is equal to $\zz$ or $-\zz$, any $\yy$ comparable with $\xx$ is such that $|\lambda(\yy)|\neq\gamma$.
Hence, the map $\lambda'$ still satisfies the second condition 
of Lemma~\ref{lem:kyfan_oct}. Since it satisfies also clearly the first condition, the conclusion of Lemma~\ref{lem:kyfan_oct} holds for $\lambda'$ and $\ds\sum_{\xx\in \Gamma}\alpha(\xx,\lambda')$ is an odd integer. It implies
$$\ds\sum_{\xx\in \Gamma}\alpha(\xx,\lambda) \equiv\ds\sum_{\xx\in \Gamma}\alpha(\xx,\lambda')\quad (\mod 2).$$ Since $\alpha(\xx,\lambda)=\alpha(\xx,\lambda')$ for every $\xx\in \Gamma\setminus\{\zz,-\zz\}$ we have
$$\alpha(\zz,\lambda)+\alpha(-\zz,\lambda)\equiv\alpha(\zz,\lambda')+\alpha(-\zz,\lambda')\quad (\mod 2).$$
We have $\alpha(-\zz,\lambda)=\alpha(\zz,\lambda')=0$ because $\lambda(-\zz)=\lambda'(\zz)$ is not of the right sign. Since $\alpha(\zz,\lambda)$ is odd, we have $\alpha(-\zz,\lambda')>0$.
The inequality $\alpha(\zz,\lambda)>0$ implies that there exists a chain $\xx_{1}\preceq\cdots\preceq\xx_{n}$ such that $\xx_{\gamma}=\zz$ and  $\lambda(\xx_{i})=(-1)^ii\quad \mbox{for all $i$}$.  The inequality $\alpha(-\zz,\lambda')>0$ implies that there exists a chain $\yy_{1}\preceq\cdots\preceq\yy_{n}$ such that $\yy_{\gamma}=-\zz$ and  $\lambda'(\yy_{i})=(-1)^ii\quad \mbox{for all $i$}$.
We have $\xx_{\gamma}=-\yy_{\gamma}=\zz$ and $\lambda'(\yy_{i})=\lambda(\yy_{i})=(-1)^ii$ for all $i\in[n]\setminus\{\gamma\}$.
Therefore, the two chains
$$\xx_{1}\preceq\cdots\preceq\xx_{n}\quad\mbox{and}\quad\yy_{1}\preceq\cdots\preceq\yy_{n}$$ are the desired chains.
\end{proof}

\section{A ``path'' of heterochromatic subgraphs}\label{sec:zigzag}

The exact statement of the zig-zag theorem goes as follows.

\begin{theorem}[Zig-zag theorem~\cite{SiTa06}]\label{thm:zigzag}
Let $G$ be a properly colored graph. Assume that the
colors are linearly ordered. Then $G$ contains a heterochromatic copy of $K_{\lceil\frac t 2\rceil,\lfloor \frac t 2\rfloor}$, where $t=\coind(B_0(G))+1$,
such that the colors appear alternating on the two sides of the bipartite subgraph with respect to their order.
\end{theorem}

Simonyi, Tardif, and Zsb\'an~\cite{SiTaZs13} showed that the statement remains true when $t=\Xind(\Hom(K_2,G)) + 2$. It is a generalization since we have the inequality $\Xind(\Hom(K_2,G)) + 2\geq\coind(B_0(G))+1$, which can moreover be strict.

A subgraph $H$ of a properly colored graph $G$ is {\em almost heterochromatic} if it contains at least $|V(H)|-1$ colors. A bipartite graph is {\em balanced} if the size of each side differs by at most two. We remind the reader that two complete bipartite subgraphs are {\em adjacent} if one is obtained from the other by removing a vertex and then adding a vertex, possibly the same one if it is added to the other side of the bipartite graph. The following theorem is our generalization of Theorem~\ref{thm:zigzag}.

\begin{theorem}\label{thm:path_colorful}
Any properly colored graph contains a sequence of adjacent almost heterochromatic complete and balanced bipartite subgraphs with $\coind(B_0(G))+1$ vertices, which starts at a heterochromatic subgraph and ends at the same heterochromatic subgraph with the two sides interchanged.
\end{theorem}

Actually, the proof also shows that the bipartite subgraph on which the sequence starts and ends has the same property as stated in Theorem~\ref{thm:zigzag}, namely that the colors appear alternating on the two sides with respect to their order (assuming the colors being linearly ordered). Note that without the statement about the interchange of the two sides, Theorem~\ref{thm:path_colorful} would be a consequence of Theorem~\ref{thm:zigzag}, since the sequence with only one subgraph would satisfy the conclusion of the theorem.

In other words, under the condition of the theorem, there is a sequence $(A_1,B_1),\ldots,(A_s,B_s)$ of pairs of disjoint subsets of $V(G)$, such that $A_1=B_s$ and $A_s=B_1$ with $\big||A_1|-|B_1|\big|\leq 1$, and such that for every $i$
\begin{itemize}
\item $G[A_i,B_i]$ is complete
\item $|c(A_1\cup B_1)|=t$
\item $|A_i\cup B_i|=t$
\item $|A_i\setminus A_{i+1}|+|B_i\setminus B_{i+1}|=|A_{i+1}\setminus A_i|+|B_{i+1}\setminus B_i|=1$ (adjacency)
\item $\big||A_i|-|B_i|\big|\leq 2$ (balancedness)
\item $c(A_i\cup B_i)\geq t-1$ (almost heterochromaticity),
\end{itemize} where $c$ is the proper coloring of $G$ and $t=\coind(B_0(G))+1$.

The proof of Theorem~\ref{thm:path_colorful} relies on the following lemma. As already defined in Section~\ref{subsec:kyfan}, a $d$-dimensional simplex $\sigma$ with a labeling $\lambda$ is negative-alternating if $\lambda(V(\sigma))$ is of the form $$\{-j_1,+j_2,\ldots,(-1)^{d+1}j_{d+1}\}$$ for some $1\leq j_1<\cdots<j_{d+1}$. It is {\em positive-alternating} if it is of the same form, with the $-$'s and $+$'s interchanged. It is {\em alternating} if it is negative-alternating or positive-alternating. It is {\em almost-alternating} if it has a facet (a $(d-1)$-dimensional face) that is alternating, while not being itself alternating. 

Assuming that $\S^d$ is embedded in $\R^{d+1}$, the two {\em hemispheres} of $\S^d$ are the sets of points of $\S^d$ whose last coordinate is respectively nonnegative and nonpositive. The {\em equator} is the intersection of these two hemispheres. (In other words, the equator is the set of points whose last coordinate is equal to $0$). A triangulation of $\S^d$ {\em refines the hemispheres} if each simplex is contained in at least one of the hemispheres of $\S^d$.

We follow the terminology of Schrijver's book~\cite{Sch03}: a {\em circuit} is a connected graph whose vertices are all of degree $2$. 

\begin{lemma}\label{lem:gen_kyfan}
Consider a centrally symmetric triangulation $\T$ of $\S^d$ that refines the hemispheres, with an antipodal labeling $\lambda:V(\T)\rightarrow\{\pm 1,\ldots,\pm m\}$ such that no adjacent vertices get opposite labels. Let $H$ be the graph whose vertices are the barycenters of the almost-alternating and alternating $d$-simplices and whose edges connect vertices if the corresponding simplices share a common alternating facet. Then $H$ contains a centrally symmetric circuit with at least two alternating $d$-simplices.
\end{lemma}

\begin{proof}
An alternating $(d-1)$-simplex is the facet of two $d$-simplices that are alternating or almost-alternating. An alternating or almost-alternating $d$-simplex has exactly two facets that are alternating.
Thus $H$ is a collection of vertex disjoint circuits. Denote by $\nu$ the central symmetry. Suppose for a contradiction that for every circuit $C$, the circuits $C$ and $\nu(C)$ are distinct. Let $o^-(C)$ (resp. $o^+(C)$) be the number of negative-alternating $(d-1)$-simplices (resp. positive-alternating) crossed by $C$ on the equator. Each circuit intersects the equator an even number of times (this number being possibly $0$), hence $o^-(C)+o^+(C)$ is even. Since $o^-(\nu(C))=o^+(C)$, we have an even number of negative-alternating $(d-1)$-simplices of the equator contained in $C$ and $\nu(C)$ together. Since two such circuits are disjoint and since an alternating $(d-1)$-simplex is contained in exactly one circuit, we have an even number of negative-alternating $(d-1)$-simplices on the equator. It contradicts Theorem~\ref{thm:kyfan} applied on the equator, which is homeomorphic to $\S^{d-1}$.

There is thus a circuit which is equal to its image by $\nu$. This circuit contains at least two centrally symmetric alternating $d$-simplices, since the sign of the smallest label of the almost-alternating $d$-simplices must change along it.
\end{proof}

\begin{proof}[Proof of Theorem~\ref{thm:path_colorful}]
Let $c$ be a proper coloring of $G$. Let $t=\coind(B_0(G))+1$. There is a $\Z_2$-map from 
$\S^{t-1}$ to $B_0(G)$. Thus there exists a simplicial $\Z_2$-map 
$\mu:\T\rightarrow B_0(G)$ for some centrally symmetric triangulation $\T$ of 
$\S^{t-1}$ that refines the hemispheres (according to the equivariant simplicial approximation 
theorem, mentioned for instance in Theorem 2 of the paper~\cite{DeLZi06} by De 
Longueville and \v{Z}ivaljevi\'c). Now, for a vertex $v$ of $\T$, define $\lambda(v)$ to be 
$+c(\mu(v))$ if $\mu(v)$ is in the first copy of $V(G)$ and to be $-c(\mu(v))$ if $\mu(v)$ is in 
the second copy of $V(G)$. The map $\lambda$ satisfies the condition of 
Lemma~\ref{lem:gen_kyfan}. There is thus a centrally symmetric circuit $C$ containing 
two centrally symmetric alternating $(t-1)$-simplices $\tau$ and $\nu(\tau)$. 
Now, start at $\tau$, and follow the circuit in an arbitrary direction, until reaching 
$\nu(\tau)$. Consider the images by $\mu$ of the almost-alternating $(t-1)$-simplices met 
and keep only those that are $(t-1)$-dimensional. This is the sought sequence of 
bipartite graphs.
\end{proof}

\section{Colorful subgraphs and cross-index}\label{sec:colorful}

The following theorem is a positive answer to the question by Simonyi, Tardif, and Zsb\'an mentioned in the introduction.

\begin{theorem}\label{thm:xind}
Let $G$ be a graph for which $$\chi(G) = \Xind(\Hom(K_2,G))+2 = t.$$  Assume that $G$ is properly colored with $[t]$ as the color set and let $I,J\subseteq[t]$ form a bipartition of the color set, i.e., $I\cup J =[t]$ and $I\cap J =\varnothing$.
Then there exists a colorful copy of $K_{|I|,|J|}$ in $G$ with the colors in $I$ on one side, and the colors in $J$ on the other side.
\end{theorem}

The original colorful $K_{\ell,m}$ theorem is the same theorem with $\coind(B_0(G))+1$ in place of $\Xind(\Hom(K_2,G))+2$. A first improvement with $\ind(\Hom(K_2,G))+2$ was obtained by Simonyi, Tardif, and Zsb\'an~\cite{SiTaZs13}. Because of the inequalities~\eqref{lbchrom}, Theorem~\ref{thm:xind} implies the two first versions of the colorful $K_{\ell,m}$ theorem. It is actually not clear that we have obtained a true generalization of the version due to Simonyi, Tardif, and Zsb\'an, since it is not known whether there exists a graph $G$ with $\Xind(\Hom(K_2,G))$ strictly larger than $\ind(\Hom(K_2,G))$.

Let $P$ be a free $\Z_2$-poset and let $\phi:P\rightarrow Q_s$ be any map for some positive integer $s$. We define an {\em alternating chain} as a chain $p_1<_P\cdots<_Pp_k$ with alternating signs, i.e., the signs of $\phi(p_i)$ and $\phi(p_{i+1})$ differ for every $i\in[k-1]$.

\begin{lemma}\label{lem:alter}
Suppose that $\phi:P\rightarrow Q_s$ is an order-preserving $\Z_2$-map. Then there exists in $P$ at least one alternating chain of length $\Xind(P)+1$. Moreover, if $s=\Xind(P)$, then for any $(s+1)$-tuple $(\varepsilon_1,\ldots,\varepsilon_{s+1})\in \{+,-\}^{s+1}$, there exits at least one chain $p_1<_P\cdots<_Pp_{s+1}$ such that $\phi(p_i)=\varepsilon_i i$ for each $i\in[s+1]$.
\end{lemma}
\begin{proof}
First, we prove that there exists at least one alternating chain of length $\Xind(P)+1$. For a contradiction, we suppose that the length of a longest alternating chain is at most $\Xind(P)$. For each $p\in P$, let $\ell(p)$ be the length of a longest alternating chain ending at $p$.
The function $\ell$ takes its value between $1$ and $\Xind(P)$.
Define the map $\bar{\phi}:P\rightarrow Q_{\Xind(P)-1}$ by
$\bar\phi(p)=\pm \ell(p)$ with the sign of $\phi(p)$. The map $\bar\phi$ is an order-preserving $\Z_2$-map from  $P$ to  $Q_{\Xind(P)-1}$, which is in contradiction with the definition of $\Xind(P)$.
This completes the proof of the first part.

Assume now that $s=\Xind(P)$. To complete the proof we shall prove that for any $(s+1)$-tuple $(\varepsilon_1,\ldots,\varepsilon_{s+1})\in \{+,-\}^{s+1}$, there exists at least one chain $p_1<_P\cdots<_Pp_{s+1}$ such that $\phi(p_i)=\varepsilon_i i$ for each $i\in[s+1]$.
To this end, define $\phi':P\rightarrow Q_s$ such that for each $p\in P$,
$\phi'(p)=(-1)^i\varepsilon_i\phi(p)$ where $i=|\phi(p)|$.
The map $\phi'$ is an order-preserving $\Z_2$-map.
In view of the prior discussion, there is an alternating chain of length $s+1$ in $P$ with respect to the map $\phi'$ and which starts with a $-$ sign. If we consider this chain with respect to the map $\phi$, one can see that this chain is the desired one.
\end{proof}

\begin{proof}[Proof of Theorem~\ref{thm:xind}] Let $c:V(G)\rightarrow[t]$ be a proper coloring of $G$. When $I$ or $J$ is empty, we have the assertion. Therefore,
without loss of generality, we may assume that $1\in I$  and $2\in J$ (we can rename colors if necessary). Define $\phi:\Hom(K_2, G)\rightarrow Q_{t-2}$ so that for
each $(A,B)\in\Hom(K_2, G)$, we have $\phi(A,B)=\pm\big(\max \left(c(A)\cup c(B)\right)-1\big)$ and the sign is positive if $\max \left(c(A)\cup c(B)\right)\in c(A)$ and is negative otherwise. The map $\phi$ is an order-preserving $\Z_2$-map. For each $i\in[t-1]$, set
$$\varepsilon_i=\left\{\begin{array}{ll}
+ & \mbox{if $i+1\in I$}\\
- & \mbox{if $i+1\in J$.}
\end{array}\right.$$
According to Lemma~\ref{lem:alter}, there is a chain of pairs of nonempty disjoint vertex subsets $(A_1,B_1)\subset \cdots\subset (A_{t-1},B_{t-1})$
such that $\phi(A_i,B_i)=\varepsilon_i i$. It implies that for $j\in I\setminus\{1\}$, we have $j\in c(A_{j-1})$, and thus $j\in c(A_{t-1})$. It implies similarly that for $j\in J$, we have $j\in c(B_{j-1})\subseteq c(B_{t-1})$. 
In particular, we have $2\in c(B_1)$. Since $\varepsilon_1=-$, we have $\max c(A_1)<2$, and thus $1\in c(A_1)$. Therefore $I\subseteq c(A_{t-1})$ and $J\subseteq c(B_{t-1})$. As $I,J$ forms a partition of $[t]$, we get $c(A_{t-1})=I$ and $c(B_{t-1})=J$, which completes the proof.
\end{proof}


\subsection*{Remark}
Lemma~\ref{lem:alter} leads to a generalization of Fan's lemma (Theorem~\ref{thm:kyfan}) which holds for any free simplicial $\Z_2$-complex. The generalization is the following (without the oddness assertion).

\begin{proposition}
Let $\K$ be a free simplicial $\Z_2$-complex and let $\lambda$ be an antipodal labeling with no complementary edges. Then there exists at least one alternating simplex of dimension $\ind(\K)$.
\end{proposition}

We explain how Lemma~\ref{lem:alter} implies this generalization. It is worth noting that this proof does not require any induction. Apply it for $P=\sd\K$ and $\phi(\sigma)=\max_{v\in V(\sigma)}\lambda(v)$, where the maximum is taken according to the partial order $\leq_{Q_s}$. The map $\phi$ is an order-preserving $\Z_2$-map from $P$ to $Q_s$. Therefore, there is an alternating chain $\sigma_1\subset\cdots\subset\sigma_{\ell}$ in $P$ of length $\ell=\Xind(\sd\K)+1\geq\ind(\sd^2\K)+1=\ind(\K)+1$. It is then easy to check that $\sigma_{\ell}$ is an alternating simplex of dimension $\ind(\K)$.

It is not too difficult to find instances for which the number of negative-alternating simplices of dimension $\ind(\K)$ is even. We can thus not expect a full generalization with the oddness assertion.

\section{Generalization of the alternative Kneser coloring lemma}\label{sec:alternative}

\subsection{The generalization}

A hypergraph $\H$ is {\em nice} if it is nonempty (i.e., it has at least one edge), if it has no singletons, and if there is a bijection $\sigma:[n]\rightarrow V(\H)$ for which
\begin{itemize}
\item $\chi(\KG(\H)) = n-\alt_{\sigma}(\H)$
\item every sign vector $\xx\in\{+,-,0\}^n$ with $\alt(\xx)\geq \alt_{\sigma}(\H)$ and $|\xx|> \alt_{\sigma}(\H)$ is such that at least one of $\sigma(\xx^+)$ and $\sigma(\xx^-)$ contains some edge of $\H$.
\end{itemize}
We use the notation $|\xx|$ to denote the cardinality of the support $\xx$, i.e., the number of $x_i$ being nonzero. Note that the chromatic number of a general Kneser graph built upon a nice hypergraph matches the lower bound of \eqref{eq:alter}. Even if the definition of a nice hypergraph is quite technical, there are some special cases that are easy to figure out, see Lemmas~\ref{lem:cd2} and~\ref{lem:PartitionMatroid} below.

The {\em categorical product} of graphs $G_1,\ldots,G_s$, denoted by $G_1\times\cdots\times G_s$, is the graph with vertex set $V(G_1)\times\cdots\times  V(G_s)$ and such that two vertices $(u_1,\ldots,u_s)$ and $(v_1,\ldots,v_s)$ are adjacent if  $u_iv_i\in E(G_i)$ for each $i$. It is a widely studied graph product, see for instance~\cite{He79,Lo67,Lo71,Ta08,Zh98}.

\begin{theorem}\label{thm:main}
Let $\H_1,\ldots,\H_s$ be nice hypergraphs and let $t=\min_j\{\chi(\KG(\H_j))\}$.
Then any proper coloring of $\KG(\H_1)\times\cdots\times\KG(\H_s)$ with $t$ colors
contains a colorful copy of $K^*_{t,t}$ with all colors appearing on each side.
\end{theorem}

The quantity $t=\min_j\{\chi(\KG(\H_j))\}$ is actually the chromatic number of $\KG(\H_1)\times\cdots\times\KG(\H_s)$ according
to a result in~\cite{HaMe16}.
We postpone the proof of Theorem~\ref{thm:main} to the end of the section. The following two lemmas give sufficient conditions for a hypergraph to be nice.

\begin{lemma}\label{lem:cd2}
Any nonempty hypergraph $\H$ with $\chi(\KG(\H))=\cd_2(\H)$ and no singletons is nice.
\end{lemma}

\begin{proof}
Let $|V(\H)|=n$ and consider an arbitrary bijection $\sigma:[n]\rightarrow V(\H)$.
The inequalities $\chi(\KG(\H))\geq n-\alt_\sigma(\H)\geq \cd_2(\H)$ hold, see Section~\ref{subsec:alt_cd}. We have thus $\chi(\KG(\H))=n-\alt_\sigma(\H)$.
Consider now a sign vector $\xx\in\{+,-,0\}^n$ with $|\xx|>n-\alt_{\sigma}(\H)$. Since $n-\alt_\sigma(\H)=\cd_2(\H)$ and since $|\xx|=|\sigma(\xx^+)|+|\sigma(\xx^-)|$, we get $|\sigma(\xx^+)|+|\sigma(\xx^-)|>n-\cd_2(\H)$. Therefore at least one of $\sigma(\xx^+)$ and $\sigma(\xx^-)$ contains some edge of $\H$, which completes the proof.
\end{proof}

\begin{lemma}\label{lem:PartitionMatroid}
Let $U_1,\ldots,U_m$ be a partition of $[n]$ and let $k,r_1,\ldots,r_m$ be positive integers. Assume that $|U_i|\neq 2r_i$ for every $i$ and that $k\geq 2$. Let $\H$ be the hypergraph defined by
$$\begin{array}{rcl}
V(\H) & = & [n] \\
E(\H) & = & \ds{\left\{A\in {[n]\choose k}:\;|A\cap U_i|\leq r_i\;\mbox{for every $i$}\right\}}.
\end{array}\footnote{The edges of such a hypergraph are the bases of a truncation of a partition matroid.}$$
If $\H$ has at least two disjoint edges, then it is nice.
\end{lemma}

The proof of Lemma~\ref{lem:PartitionMatroid} is long and technical. To ease the reading of the paper, we omit it. It can however be found at the following address:

{\footnotesize\noindent\url{http://facultymembers.sbu.ac.ir/hhaji/documents/Publications_files/Partition-Matroids-Are-Nice.pdf}.}

Theorem~\ref{thm:main} combined with Lemma~\ref{lem:cd2} shows in particular that if a nonempty hypergraph $\H$ has no singletons and is such that $\chi(\KG(\H)=\cd_2(\H)$, then we have the existence of a colorful copy of $K^*_{t,t}$. This is exactly the statement of Corollary~\ref{cor:cd2}. We have already mentioned that this corollary implies immediately the alternative Kneser coloring lemma with the help of the Lov\'asz theorem. Actually, the alternative Kneser coloring lemma can also be obtained from Theorem~\ref{thm:main} (with $s=1$) and Lemma~\ref{lem:PartitionMatroid} together (either by taking $m=n$ and $r_i=|U_i|=1$ for every $i$, or by taking $m=1$ and $r_1=k$).


\subsection{Circular chromatic number}\label{subsec:circ}

Let $G=(V,E)$ be a graph. For two integers $p\geq q\geq 1$, a {\em $(p,q)$-coloring} of $G$ is a mapping $c:V\rightarrow[p]$ such that $q\leq|c(u)-c(v)|\leq p-q$ for every edge $uv$ of $G$. The {\em circular chromatic number} of $G$ is $$\chi_c(G)=\inf\{p/q:\; G\mbox{ admits a $(p,q)$-coloring}\}.$$ The inequalities $\chi(G)-1<\chi_c(G)\leq\chi(G)$ hold. Moreover, the infimum in the definition is actually a minimum, i.e., $\chi_c(G)$ is attained for some $(p,q)$-coloring (and thus the circular chromatic number is always rational), see \cite{Zh01} for details. The question of determining which graphs $G$ are such that $\chi_c(G)=\chi(G)$ has received a considerable attention~\cite{Zh99,Zhusurvey2006}. In particular, a conjecture by Johnson, Holroyd, and Stahl~\cite{JHS1997} stated that the circular chromatic number of $\KG(n,k)$ is equal to its chromatic number.  It is known (cf.~\cite{CLZ13} or ~\cite{HaTa10}) and easy to prove by the definition of circular coloring that if $G$ is $t$-colorable and every proper coloring with $t$ colors of $G$ contains a $K^*_{t,t}$ with all $t$ colors appearing on each side, then $\chi_c(G)=\chi(G)$. With his colorful theorem, Chen was able to prove the Johnson-Holroyd-Stahl conjecture, after partial results obtained by Hajiabolhassan and Zhu~\cite{HaZh03}, 
Meunier~\cite{Me05}, and Simonyi and Tardos~\cite{SiTa06}. Corollary~\ref{cor:cd2} implies the following result.

\begin{corollary}\label{cor:circ}
Let $\H$ be a nonempty hypergraph such that $\chi(\KG(\H))=\cd_2(\H)$. Then $\chi(\KG(\H))=\chi_c(\KG(\H))$.
\end{corollary}

%
%

If $\H$ has a singleton, it is not nice and we cannot apply Theorem~\ref{thm:main}. Nevertheless, its conclusion holds since $\KG(\H)$ is then homomorphic to a graph 
with a vertex adjacent to any other vertex -- a so-called {\em universal} vertex -- and it is known that every 
graph with a universal vertex has its chromatic number equal to its circular chromatic 
number \cite{Gu93,Zh92bis}.

Note moreover that Theorem~\ref{thm:main} implies that all these graphs satisfy the Hedetniemi conjecture for the circular chromatic number, proposed by Zhu~\cite{Zh92bis}.

\subsection{Proof of Theorem~\ref{thm:main}}

The original proof of the alternative Kneser coloring lemma given by Chen~\cite{Ch11} was simplified by Chang, Liu, and Zhu~\cite{CLZ13}. A further simplification was subsequently obtained by Liu and Zhu~\cite{LZ15}. In our approach, we reuse ideas proposed in these papers, as well as techniques developed by the last two authors to deal with the categorical product of general Kneser graphs~\cite{HaMe16}. 

The proof relies on Lemma~\ref{lem:chen} in a crucial way but it is rather intriguing that the argument does not work when $\min_{j\in[s]}\chi(\KG(\H_j))\leq 2$. We split therefore the proof into two parts: the first part is about the case $\min_{j\in[s]}\chi(\KG(\H_j))\leq 2$ and the second part is about the case $\min_{j\in[s]}\chi(\KG(\H_j))\geq 3$. The first part is rather straightforward and does not use other results.

Within the proof, for a nice hypergraph $\H_j$, we denote by $n_j$ the cardinality of $V(\H_j)$ and by $\sigma_j$ the bijection $[n_j]\rightarrow V(\H_j)$ whose existence is ensured by the niceness of $\H_j$. Note that because of the inequality~\eqref{eq:alter}, we have $\alt_{\sigma_j}(\H_j)=\alt(\H_j)$ for every $j$. Moreover, we identify each element of $V(\H_j)$ with its preimage by $\sigma_j$. Doing this, we get that for every $j$, the set $V(\H_j)$ is $[n_j]$ and that $\sigma_j$ is the identity permutation. Thus, if $\xx\in\{+,-,0\}^{n_j}$ is such that neither $\xx^+$ nor $\xx^-$ contains an edge of $\H_j$, then we have $\alt(\xx)\leq\alt(\H_j)$. This will be used throughout the proof without further mention.

We set $n:=\sum_jn_j$.

\begin{proof}[Proof of Theorem~\ref{thm:main} when $\min_{j\in[s]}\chi(\KG(\H_j))\leq 2$]Suppose first that $\min_{j\in[s]}\chi(\KG(\H_j))=1$ holds, i.e., at least one of the $\KG(\H_j)$ is a stable set, say $\KG(\H_1)$ without loss of generality. Since $\H_1$ is nice and $\chi(\KG(\H_1))=1$, we have $\alt(\H_1)=n_1-1$. Hence, according to the definition of $\alt(\H_1)$, there exists an $\xx\in\{+,-,0\}^{n_1}$ with no edge in $\xx^+$ and no edge in $\xx^-$ and such that $\alt(\xx)=|\xx|=n_1-1$. Moreover, this $\xx$ is such that making the $0$ entry of $\xx$ a $+$ creates an edge $e^+$ of $\H_1$ in $\xx^+$, and similarly making this $0$ a $-$ creates  an edge $e^-$ in $\xx^-$. These two edges $e^+$ and $e^-$ are distinct because $\H_1$ has no singletons. Since $\KG(\H_1)\times\cdots\times\KG(\H_s)$ is in this case also a stable set, with at least two vertices whose existence is ensured by $e^+$ and $e^-$, we are done: there is a $K_{1,1}^*$.

Let us suppose now that $\min_{j\in[s]}\chi(\KG(\H_j))=2$. Again, assume without loss of generality that the minimum is attained for $j=1$. Since $\H_1$ is nice and $\chi(\KG(\H_1))=2$, we have $\alt(\H_1)=n_1-2$. There exists thus an $\xx\in\{+,-,0\}^{n_1}$ with no edge in $\xx^+$ and no edge in $\xx^-$ and such that $\alt(\xx_1)=|\xx_1|=n_1-2$. Similarly, as before, using the $0$ entries of $\xx_1$, we can find in $\H_1$ the following four distinct edges, $e^+, e^-, f^+, f^-$ such that on the one hand, $e^+$ and $f^-$ are disjoint and the other hand $e^-$ and $f^+$ are disjoint. $\KG(\H_1)$ being bipartite, $\KG(\H_1)\times\cdots\times\KG(\H_s)$ is also bipartite, with at least two disjoint edges. There is thus a $K_{2,2}^*$, and we are done since any $2$-coloring of $\KG(\H_1)\times\cdots\times\KG(\H_s)$ provides a colorful $K_{2,2}^*$.
\end{proof}

\begin{proof}[Proof of Theorem~\ref{thm:main} when $\min_{j\in[s]}\chi(\KG(\H_j))\geq 3$]Define $t=\min_{j\in[s]}\chi(\KG(\H_j))$. Note that since the $\H_j$'s are nice, we have $t\leq n_j-\alt(\H_j)$ for every $j$. This will be used in the proof without further mention. Let $c$ be a proper coloring of $\KG(\H_1)\times\cdots\times\KG(\H_s)$ with $[t]$ as the color set, which exists because of the easy direction of Hedetniemi's conjecture. We suppose throughout the proof that $n-t$ is even. This can be done without loss of generality simply by adding a dummy vertex $n_j+1$ to any of the hypergraphs $\H_j$'s. Doing this, we do~not change the general Kneser graphs and thus $t$ remains the same, $n$ increases by exactly one, and the modified $\H_j$ remains nice.

We will define a map $\lambda:\{+,-,0\}^n\setminus \{\zero\}\rightarrow\{\pm 1,\ldots,\pm n\}$ for which we will apply Chen's lemma (Lemma~\ref{lem:chen}). To do this, we need to introduce some notations. For $\xx\in\{+,-,0\}^n\setminus \{\zero\}$, we define $\xx(1)\in\{+,-,0\}^{n_1}$ to be the  first $n_1$ coordinates of $\xx$, $\xx(2)\in\{+,-,0\}^{n_2}$ to be the next $n_2$ coordinates of $\xx$, and so on, up to $\xx(s)\in\{+,-,0\}^{n_s}$ to be the last $n_s$ coordinates of $\xx$. We define moreover $A_j(\xx)$ to be the set of signs $\varepsilon\in\{+,-\}$ such that $\xx(j)^{\varepsilon}$ contains at least one edge of $\H_j$. Let us just come back to the notation, to avoid any ambiguity: $\xx(j)^{\varepsilon}$ is the set of $i\in[n_j]$ such that $x_{i+\sum_{j'=1}^{j-1}n_{j'}}=\varepsilon$.

We define $\lambda(\xx)$ by defining its sign $s(\xx)\in\{+,-\}$ and its absolute value $v(\xx)$. In other words, $\lambda(\xx):=s(\xx)v(\xx)$. Two cases have to be distinguished. We proceed similarly as in \cite{HaMe16}. \\

\noindent {\bf First case:} $\bigcap_j A_j(\xx)=\varnothing$.

Set
$$\begin{array}{rcl}
v(\xx) & = & \ds{\sum_{j:\;|A_j(\xx)|=0}\alt(\xx(j))+\sum_{j:\;|A_j(\xx)|=2}|\xx(j)|} \\ \\

& & +\ds{\sum_{j:\;|A_j(\xx)|=1}\left(1+\max\{\alt(\widetilde{\yy}):\;\widetilde{\yy}\preceq\xx(j)\mbox{ and }E(\H_j[\widetilde\yy^+])=E(\H_j[\widetilde\yy^-])=\varnothing\}\right)}, \end{array}$$
where $\H_j[\widetilde\yy^+]$ (resp. $\H_j[\widetilde\yy^-]$) denotes the restriction of $\H_j$ to $\widetilde\yy^+$ (resp. $\widetilde\yy^-$).
In this formula, for $j$ with $A_j(\xx)$ of cardinality one, we are looking for a $\widetilde\yy\preceq\xx(j)$ with the longest alternating subsequence such that neither $\widetilde\yy^+$, nor $\widetilde\yy^-$, contains an edge of $\H_j$. The sign $s(\xx)$ is defined to be the first nonzero coordinate of $\xx$ if none of the $A_j(\xx)$'s is of cardinality one, and to be the sign in the $A_j(\xx)$ of cardinality one with the smallest possible $j$, otherwise. \\

\noindent {\bf Second case:} $\bigcap_j A_j(\xx)\neq\varnothing$.

 Define
$$c^{\varepsilon}(\xx)=\max\{c(e_1,\ldots,e_s):\;e_j\subseteq\xx(j)^{\varepsilon}\;\mbox{and}\;e_j\in E(\H_j)\;\mbox{for all $j\in[s]$}\}$$ where it takes the value $-\infty$ if there is no such $s$-tuple $(e_1,\ldots,e_s)$.
Define moreover $c(\xx)=\max\left(c^+(\xx),c^-(\xx)\right)$.
Since $\bigcap_j A_j(\xx)\neq\varnothing$, at least one of $c^+(\xx)$ and $c^-(\xx)$ is finite. Now, set $$v(\xx)=n-t+c(\xx).$$
The sign $s(\xx)$ is $+$ if $c^+(\xx)>c^-(\xx)$, and $-$ otherwise. Note that since $c$ is a proper coloring, we have $c^+(\xx)\neq c^-(\xx)$.

\begin{claim}\label{claim:vt}
Let $\xx\in\{+,-,0\}^n\setminus\{\zero\}$. We have $$\bigcap_j A_j(\xx)=\varnothing\quad\Longleftrightarrow\quad v(\xx)\leq n-t.$$ Moreover, if $v(\xx)=n-t$, then there exists $j_0\in[s]$ such that
\begin{itemize}
\item $A_{j_0}(\xx)=\varnothing$ and $\alt(\xx(j_0))=|\xx(j_0)|=\alt(\H_{j_0})=n_{j_0}-t$,
\item $|A_j(\xx)|=2$ and $|\xx(j)|=n_j\quad\forall j\neq j_0$.
\end{itemize}
\end{claim}

\begin{proof-claim}
Let $\xx\in\{+,-,0\}^n\setminus\{\zero\}$. Let us prove the equivalence. If $\bigcap_j A_j(\xx)\neq\varnothing$, then according to the definition of $\lambda$, we necessarily have $v(\xx)\geq n-t+1$. Now, suppose that $\bigcap_j A_j(\xx)=\varnothing$. At least one set $A_j(\xx)$ is not of cardinality two, otherwise $\bigcap_jA_j(\xx)$ would be nonempty. If there are at least two sets $A_j(\xx)$ not of cardinality two, say $A_{j_1}(\xx)$ and $A_{j_2}(\xx)$, then the maximum value we can get for $v(\xx)$ is at most $n+\sum_{\ell=1}^2(1+\alt(\H_{j_{\ell}})-n_{j_\ell})$, which is at most $n-2t+2$, which is itself at most $n-t-1$ since we suppose $t\geq 3$. If there is only one such set, say $A_{j_0}(\xx)$, it is necessarily empty (otherwise $\bigcap_jA_j(\xx)$ would be again nonempty), each $A_j(\xx)$ is of cardinality two except for $j=j_0$, and we have
$$v(\xx)=\alt(\xx({j_0}))+\sum_{j\neq j_0}|\xx(j)|\leq\alt(\H_{j_0})+\sum_{j\neq j_0}n_j\leq n-t,$$ which finishes the proof of the equivalence.

We have actually proved that if $v(\xx)=n-t$, then the sets $A_j(\xx)$ are all of cardinality two for $j\neq j_0$, the set $A_{j_0}(\xx)$ is empty, and
$$\alt(\xx({j_0}))+\sum_{j\neq j_0}|\xx(j)|=\alt(\H_{j_0})+\sum_{j\neq j_0}n_j=n-t.$$
Since $\alt(\xx({j_0}))\leq\alt(\H_{j_0})$ and $|\xx(j)|\leq n_j$ by definition, these inequalities are actually equalities. The hypergraph $\H_{j_0}$ being nice, the equality $\alt(\xx({j_0}))=\alt(\H_{j_0})$ and the fact that $A_{j_0}(\xx)$ is empty imply that $|\xx(j_0)|=\alt(\xx({j_0}))$. We get thus the second part of the statement.
\end{proof-claim}

The map $\lambda$ is defined on the elements of the poset $(\{+,-,0\}^n\setminus\{\zero\},\preceq)$ and takes its values in $\{\pm 1,\ldots,\pm n\}$.

\begin{claim}
$\lambda$ satisfies the condition of Lemma~\ref{lem:chen} with $\gamma=n-t$.
\end{claim}
\begin{proof-claim}
The fact that $\lambda$ is antipodal is direct from the definition. Let us check that $\lambda$ is an order-preserving map $(\{+,-,0\}^n\setminus\{\zero\},\preceq)\rightarrow Q_{n-1}$. To do this, take $\xx$ and $\yy$ such that $\xx\preceq\yy$. We have $\bigcap_j A_j(\xx)\subseteq\bigcap_j A_j(\yy)$. If $\bigcap_j A_j(\xx)$ and $\bigcap_j A_j(\yy)$ are both empty or both nonempty, it is clear from the definition that $v(\xx)\leq v(\yy)$. If $\bigcap_j A_j(\xx)$ is empty and $\bigcap_j A_j(\yy)$ is nonempty, then according to Claim~\ref{claim:vt}, we have $v(\xx)\leq n-t$ and $v(\yy)\geq n-t+1$. This shows that $v(\xx)\leq v(\yy)$ when $\xx\preceq\yy$.

To finish the checking that $\lambda$ is order-preserving, we have to show that when $\xx\preceq\yy$, the quantity $\lambda(\xx)+\lambda(\yy)$ is not equal to $0$. Assume for a contradiction that there exist such $\xx$ and $\yy$ with $\xx\preceq\yy$ and $\lambda(\xx)+\lambda(\yy)=0$. Since $v$ is nondecreasing, we can assume that $\yy$ differs from $\xx$ only by one entry, i.e., making a $0$ entry of $\xx$ a nonzero one leads to $\yy$. Let us suppose that $\bigcap_j A_j(\xx)=\varnothing$. Since $v(\xx)=v(\yy)$, the entry that differs between $\xx$ and $\yy$ belongs to an $\xx(j)$ with $|A_j(\xx)|\neq 2$. If $|A_j(\xx)|=0$, the new $A_j(\xx)$ is still empty, and it is easy to check that the sign cannot change. If $|A_j(\xx)|=1$, the new $A_j(\xx)$ is still of cardinality one, and the sign does not change either. We see thus that the assumption implies that we necessarily have $\bigcap_j A_j(\xx)\neq\varnothing$. But then again, the sign cannot change when we go from $\xx$ to $\yy$ since $c$ is a proper coloring.

Finally, let us check that there are no $\xx\prec\yy$ such that $|\lambda(\xx)|=|\lambda(\yy)|=\gamma$. Suppose for a contradiction that such $\xx$ and $\yy$ exist. Claim~\ref{claim:vt} applied on $\xx$ and $\yy$ shows that there is a unique $j_0$ such that $A_{j_0}(\xx)$ is empty and a unique $j_0'$ such that $A_{j_0'}(\yy)$ is empty. Since $\xx\prec\yy$, we have $j_0=j_0'$. Claim~\ref{claim:vt} again shows then that $|\xx(j)|=|\yy(j)|$ for all $j\neq j_0$. Hence $\xx=\yy$, and we get a contradiction.
\end{proof-claim}

We can thus apply Lemma~\ref{lem:chen}. There are two chains
$$\xx_{1}\preceq\cdots\preceq\xx_{n}\quad\mbox{and}\quad\yy_{1}\preceq\cdots\preceq\yy_{n}$$ such that
$$\lambda(\xx_{i})=(-1)^ii\quad \mbox{for all $i$}\qquad\mbox{and }\qquad\lambda(\yy_{i})=(-1)^ii \quad\mbox{for $i\neq\gamma$}$$
and such that $\xx_{\gamma}=-\yy_{\gamma}$, with $\gamma=n-t$. We use now $j_0$ as the element of $[s]$ whose existence is ensured by Claim~\ref{claim:vt} applied on $\xx=\xx_{\gamma}$.

\begin{claim}\label{claim:prog}
On the one hand, we have $\xx_{\gamma}(j)=\cdots=\xx_n(j)$ for every $j\neq j_0$. On the other hand, $\xx_{i+1}(j_0)$ is obtained from $\xx_i(j_0)$ by replacing exactly one of its zero entries by a nonzero one, for $i=\gamma,\ldots,n-1$. The same assertion holds for the $\yy_i$'s.
\end{claim}

\begin{proof-claim}
According to Claim~\ref{claim:vt}, $\xx_{\gamma}(j)$ has no zero entries for $j\neq j_0$. Because of the chain $\xx_{\gamma}\preceq\cdots\preceq\xx_n$, we get the first part of the statement.  According to Claim~\ref{claim:vt} again, the number of zero entries of $\xx_{\gamma}(j_0)$ is exactly $t$ since $|\xx_{\gamma}(j_0)|=n_{j_0}-t$. Since $\xx_{\gamma},\ldots,\xx_n$ get distinct images by $\lambda$ and since $\xx_{\gamma}(j),\ldots,\xx_n(j)$ do not differ when $j\neq j_0$, all $\xx_{\gamma}(j_0),\ldots,\xx_n(j_0)$ are pairwise distinct. The second part of the statement follows from $n-\gamma=t$.

The assertion for the $\yy_i$'s is proved similarly.
\end{proof-claim}

We define $S$ to be $\xx_{\gamma}(j_0)^+$ and $T$ to be $\xx_{\gamma}(j_0)^-$. Note that $S$ and $T$ are disjoint and that none of them contain an edge of $\H_{j_0}$.

\begin{claim}\label{claim:defab}
There exist integers $a_{\gamma+1},\ldots,a_n,b_{\gamma+1},\ldots,b_n\in[n_{j_0}]\setminus(S\cup T)$ such that
\begin{itemize}
\item for odd $i\geq\gamma+1$:
$$\xx_i(j_0)^-=T\cup\{a_{\gamma+1},a_{\gamma+3},\ldots,a_i\}\quad\mbox{and}\quad\yy_i(j_0)^-=S\cup\{b_{\gamma+1},b_{\gamma+3},\ldots,b_i\}$$
\item for even $i\geq\gamma+2$:
$$\xx_i(j_0)^+=S\cup\{a_{\gamma+2},a_{\gamma+4},\ldots,a_i\}\quad\mbox{and}\quad\yy_i(j_0)^+=T\cup\{b_{\gamma+2},b_{\gamma+4},\ldots,b_i\}.$$
\end{itemize}
Moreover, the $a_i$'s are pairwise distinct and so are the $b_i$'s.
\end{claim}
\begin{proof-claim}
According to Claim~\ref{claim:prog}, $\xx_{i-1}(j_0)$ and $\xx_i(j_0)$ differ by exactly one entry for $i\in\{\gamma+1,\ldots,n\}$. Define $a_i$ to be the entry index of $\xx_{i-1}(j_0)$ that becomes nonzero in $\xx_i(j_0)$.

We have $s(\xx_{\gamma+1})=-$ and $v(\xx_{\gamma+1})=\gamma+1$. It implies that there is an edge of $\H_{j_0}$ in $\xx_{\gamma+1}^-$ that gets the color $1$. Since $A_{j_0}(\xx_{\gamma})=\varnothing$, this edge contains necessarily $a_{\gamma+1}$ and thus $\xx_{\gamma+1}^-(j_0)=T\cup\{a_{\gamma+1}\}$.

We have $s(\xx_{\gamma+2})=+$ and $v(\xx_{\gamma+2})=\gamma+2$. It implies that there is an edge of $\H_{j_0}$ in $\xx_{\gamma+2}^+$ that gets the color $2$. Since $A_{j_0}(\xx_{\gamma+1})=\{-\}$ (this has just been proved), this edge contains necessarily $a_{\gamma+2}$ and thus $\xx_{\gamma+2}^+(j_0)=S\cup\{a_{\gamma+2}\}$.

For odd $i\geq\gamma+3$, we have $s(\xx_i)=-$ and $v(\xx_i)=i$. It implies that there is an edge of $\H_{j_0}$ in $\xx_i^-$ that gets the color $i-\gamma$. Since the map $v$ is obtained by taking the maximum possible color, we know that neither $\xx_{i-1}^-$, nor $\xx_{i-1}^+$ contain an edge with color $i-\gamma$. It implies that the $a_i$th entry of $\xx_i(j_0)$ is a $-$. The case with even $i\geq\gamma+4$ is treated similarly.

The statement for the $\yy_i(j_0)$'s is obtained with almost the same proof, once being noted that $\yy_{\gamma}(j_0)^+=T$ and $\yy_{\gamma}(j_0)^-=S$.
\end{proof-claim}

For odd $i\geq\gamma+1$, there is an edge $e_{j_0}^i\in E(\H_{j_0})$ such that $e_{j_0}^i\subseteq T\cup\{a_i\}$. Indeed, let $\zz$ be the sign vector of $\{+,-,0\}^{n_{j_0}}$ obtained from $\xx_{\gamma}(j_0)$ by making its $a_i$th entry a $-$. We have then $\alt(\zz)\geq\alt(\xx_{\gamma}(j_0))=\alt(\H_{j_0})$ and $|\zz|>\alt(\H_{j_0})$. Since $\H_{j_0}$ is nice and since neither $S$, nor $T$ contains an edge of $\H_{j_0}$, we get the claimed existence of the edge. The same holds for $S\cup\{b_i\}$: there is an edge $f_{j_0}^i\in E(\H_{j_0})$ such that $f_{j_0}^i\subseteq S\cup\{b_i\}$. For even $i\geq\gamma+2$, we have similarly the existence of edges $e_{j_0}^i,f_{j_0}^i\in E(\H_{j_0})$ such that $e_{j_0}^i\subseteq T\cup\{b_i\}$ and $f_{j_0}^i\subseteq S\cup\{a_i\}$. Note that since $\H$ has no singleton, the edges $e_{j_0}^i$ and $f_{j_0}^i$ are distinct.

Because of Claim~\ref{claim:vt}, we know also that there exist edges $e_j,f_j\in E(\H_j)$ such that $e_j\subseteq\xx_{\gamma}(j)^-$ and $f_j\subseteq\xx_{\gamma}(j)^+$ for every $j\neq j_0$, since $|A_j(\xx_{\gamma})|=2$.

We define now $2t$ vertices of $\KG(\H_1)\times\cdots\times\KG(\H_s)$. We will check that they induce a graph containing a colorful copy of $K_{t,t}^*$. For $i=\gamma+1,\ldots,n$, we define
\begin{itemize}
\item $\ee_i$ to be the $s$-tuple whose $j$th entry is $e_j$, except the $j_0$th one, which is $e_{j_0}^i$.
\item $\ff_i$ to be the $s$-tuple whose $j$th entry is $f_j$, except the $j_0$th one, which is $f_{j_0}^i$.
\end{itemize}
They are vertices of $\KG(\H_1)\times\cdots\times\KG(\H_s)$.
\begin{claim}\label{claim:aisb}
For each $i\in\{\gamma+1,\ldots,n\}$, we have $a_i=b_i$ (defined according to Claim~\ref{claim:defab}) and
$c(\ee_i)=c(\ff_i)=i-\gamma$.
\end{claim}
\begin{proof-claim}

We first prove the cases $i=\gamma+1$ and $i=\gamma+2$. Let us start with $i=\gamma+1$. By definition of $\ee_{\gamma+1}$, we have $1\leq c(\ee_{\gamma+1})\leq c(\xx^-_{\gamma+1})$. Since $\lambda(\xx_{\gamma+1})=-(\gamma+c(\xx_{\gamma+1}))=-(\gamma+1)$, the equality $c(\xx^-_{\gamma+1})=1$ holds, and hence $c(\ee_{\gamma+1})=1$. Similarly, since $\lambda(\yy_{\gamma+1})=-(\gamma+c(\yy_{\gamma+1}))=-(\gamma+1)$, the equality $c(\ff_{\gamma+1})=1$ holds. 
The fact that $\ee_{\gamma+1}$ and $\ff_{\gamma+1}$ get the same color implies moreover that they are not adjacent as the vertices of $\KG(\H_1)\times\cdots\times\KG(\H_s)$. 
Since $e_j$ and $f_j$ are disjoint for $j\neq j_0$, it implies that $e_{j_0}^{\gamma+1}\cap f_{j_0}^{\gamma+1}\neq\varnothing$, which implies that $a_{\gamma+1}=b_{\gamma+1}$, as $S$ and $T$ are disjoint.

Now, let us look at the case $i=\gamma+2$. By definition of $\ee_{\gamma+2}$, we have $1\leq c(\ee_{\gamma+2})\leq c(\yy^+_{\gamma+2})$. Since $v(\yy_{\gamma+2})=\gamma+2$, the inequality $c(\yy^+_{\gamma+2})\leq 2$ holds, and hence $c(\ee_{\gamma+2})\in\{1,2\}$. Similarly, $c(\ff_{\gamma+2})\in\{1,2\}$. Since $e_j$ and $f_j$ are disjoint for $j\neq j_0$ and since $e_{j_0}^{\gamma+2}$ and $f_{j_0}^{\gamma+1}$ are disjoint, $\ee_{\gamma+2}\ff_{\gamma+1}$ is an edge of $\KG(\H_1)\times\cdots\times\KG(\H_s)$ and thus $c(\ee_{\gamma+2})\neq 1$. Therefore $c(\ee_{\gamma+2})=2$. Similarly, $c(\ff_{\gamma+2})=2$. Now, we use the same technique as for $i=\gamma+1$ to prove that $a_{\gamma+2}=b_{\gamma+2}$: the fact that $\ee_{\gamma+2}$ and $\ff_{\gamma+2}$ get the same color implies that $e_{j_0}^{\gamma+2}$ and $f_{j_0}^{\gamma+2}$ get a nonempty intersection. This is possible only if $a_{\gamma+2}=b_{\gamma+2}$.

For the remaining values of $i$, we proceed by induction. Let $i$ be an odd integer such that $i\geq\gamma+3$. We have $1\leq c(\ee_i)\leq c(\xx^-_i)$. Since $v(\xx_i)=i$, we get on the one hand $c(\ee_i)\leq i-\gamma$.
On the other hand, the induction implies that $\{a_{\gamma+1},\ldots,a_{i-1}\}=\{b_{\gamma+1},\ldots,b_{i-1}\}$. Thus $e_{j_0}^i$ and $f_{j_0}^{i'}$ are disjoint for $i'=\gamma+1,\ldots,i-1$. Since $e_j$ and $f_j$ are disjoint for $j\neq j_0$, we get that $\ee_i$ and $\ff_{i'}$ are adjacent as the vertices of $\KG(\H_1)\times\cdots\times\KG(\H_s)$ for $i'=\gamma+1,\ldots,i-1$. Then the induction implies  that $c(\ee_i)\geq i-\gamma$. Therefore, $c(\ee_i)=i-\gamma$. Similarly, we prove $c(\ff_i)=i-\gamma$. Now, we use the same technique again to prove $a_i=b_i$: the fact that $\ee_i$ and $\ff_i$ get the same color implies that $e_{j_0}^i$ and $f_{j_0}^i$ get a nonempty intersection. This is possible only if $a_i=b_i$. The case when $i$ is an even integer such that $i\geq \gamma+4$ is dealt with similarly.
\end{proof-claim}

The vertices $\ee_i$  and $\ff_{i'}$ of $\KG(\H_1)\times\cdots\times\KG(\H_s)$ are adjacent when $i\neq i'$. Indeed, $e_j$ and $f_j$ are obviously disjoint for every $j\neq j_0$, and so are $e_{j_0}^i$ and $f_{j_0}^{i'}$, as we explain now. For odd $i$, we have $e_{j_0}^i\subseteq S\cup\{a_i\}$ while $f_{j_0}^{i'}\subseteq T\cup\{b_{i'}\}$, with $b_{i'}=a_{i'}\neq a_i$ (this is a consequence of Claim~\ref{claim:aisb}), and we have $e_{j_0}^i$ and $f_{j_0}^{i'}$ disjoint. For even $i$, we have $e_{j_0}^i\subseteq S\cup\{b_i\}$ while $f_{j_0}^{i'}\subseteq T\cup\{a_{i'}\}$, with $a_{i'}=b_{i'}\neq b_i$ (this is again a consequence of Claim~\ref{claim:aisb}), and we have $e_{j_0}^i$ and $f_{j_0}^{i'}$ disjoint, again. 
Moreover since $\ee_i$ and $\ff_i$ are distinct for every $i$, they induce a subgraph containing a $K_{t,t}^*$ in $\KG(\H_1)\times\cdots\times\KG(\H_s)$. The statement about the colors of this $K_{t,t}^*$ is also a consequence of Claim~\ref{claim:aisb}.
\end{proof}

\subsection*{Remark}
Theorem~\ref{thm:main} remains true under the weaker condition that $n_j-\alt(\H_j) \geq t$ for all $j$ and that the $\H_j$'s for which the equality holds are nice. Indeed, when $t\geq 3$, we use the fact that the $\H_j$'s are nice only twice: in the proof of Claim~\ref{claim:vt} and right after the proof of Claim~\ref{claim:defab}, and in both places $j=j_0$ (which is such that $\chi(\KG(\H_{j_0}))=t$). When $t\leq 2$, we also use the fact that the $\H_j$'s are nice only for those $j$ such that  $\chi(\KG(\H_j))=t$.



\section{Hypergraphs $\H$ such that $\chi(\KG(\H))=\cd_2(\H)$}\label{sec:cd}

Motivated by Corollary~\ref{cor:circ}, we try to better understand the hypergraphs $\H$ such that $\chi(\KG(\H))=\cd_2(\H)$. We found this question interesting for its own sake. 



\subsection{A construction}

Let $n\geq 2k-1$ and $m\geq 1$. We define the set system
$$F_{n,m,k}={[n] \choose k}\cup\big\{\{i,j\}:\;i\in[n],\;j\in[n+m]\setminus[n]\big\}\cup{[n+m]\setminus[n] \choose k}.$$

\begin{proposition}
The hypergraph $\H_{n,m,k}$ with vertex set $[n+m]$ and edge set $F_{n,m,k}$ is such that $\chi(\KG(\H_{n,m,k}))=\cd_2(\H_{n,m,k})=n+m-2k+2$.
\end{proposition}

\begin{proof}
We color $\KG(\H_{n,m,k})$ as follows. Let $A$ be one of its vertices. If $A\subseteq[n]$, we give to $A$ the color $\min(\min A, n-2k+2)$ (usual coloring of Kneser graphs). Otherwise, we give to $A$ the color $\min(A\cap[n+m]\setminus[n])$. This is clearly a proper coloring with $n+m-2k+2$ colors.

To obtain a 2-colorable hypergraph from $\H_{n,m,k}$, there are three possibilities, which we study in turn. Either we delete $[n+m]\setminus[n]$ completely, or we delete $[n]$ completely, or we delete none of them completely.

 If we delete $[n+m]\setminus[n]$ completely, we still have to delete at least $n-2k+2$ elements of $[n]$. If we delete $[n]$ completely and $m\geq 2k-1$, then we still have to delete $m-2k+2$ vertices on the other part. If we delete $[n]$ completely and $m< 2k-1$, then $n\geq n+m-2k+2$, and so we again delete at least $n+m-2k+2$ points. Now assume that we do not delete either part completely. Then we have to delete all but $k-1$ vertices on both parts, which again gives $n+m-2k+2$ vertices for deletion or more if $m<k-1$. This proves $\cd_2(\H_{n,m,k})\geq n+m-2k+2$.

Dol'nikov's theorem, i.e., Equation~\eqref{eq:dolnikov}, allows to conclude.
\end{proof}

\subsection{When $\H$ is a graph}

We provide in this section a necessary condition for a graph $G$ to be such that $\chi(\KG(G))=\cd_2(G)$, via an elementary proof of Dol'nikov's theorem (Equation~\eqref{eq:dolnikov}) in this case.

Let $G$ be a graph. A proper coloring of $\KG(G)$ is a partition of the edges of $G$ into triangles $T_i$ and stars $S_j$. The quantity $\cd_2(G)$ is the minimal number of vertices to remove from $G$ so that the remaining vertices induce a bipartite graph.

\begin{proposition}
Let $G$ be a graph such that $\chi(\KG(G))=\cd_2(G)$. Then there is an optimal proper coloring of $\KG(G)$ having at least one triangle and whose triangles $T_i$ are pairwise vertex disjoint.
\end{proposition}

\begin{proof}
Take an optimal coloring with a minimum number of triangles.
We claim that in such a coloring, every circuit is either a $T_i$, or contains at least one edge belonging to a star color class $S_j$. (Here again, by circuit, we mean a connected graph whose vertices are all of degree $2$.)

Indeed, suppose not. Take the shortest circuit $C$ contradicting this claim. Each edge of $C$ belongs to a $T_i$, and each $T_i$ has at most one edge on $C$ (this second statement is because of the minimality of $C$). We change the coloring as follows. We remove all triangles met by $C$ from the coloring, and for each vertex $v$ of $C$, put the star whose center is $v$. It provides a new coloring, with the same number of colors: the number of triangles that have been removed is equal to the number of vertices of $C$, which is equal to the number of new stars. This contradicts the minimality assumption on the number of triangles.

Now, take an optimal coloring with a minimum number of triangles.

Suppose first for a contradiction that there are no triangles. Then removing the centers of all stars but one, and all edges incident to them, leads to a graph with only one star, which is a bipartite graph.  The number of vertices that have been removed is the number of colors minus $1$, although we have obtained a bipartite graph. This is in contradiction with $\chi(\KG(G))=\cd_2(G)$.

Suppose now, still for a contradiction, that two triangles have a vertex in common.
Remove the centers of the stars $S_j$, and all edges incident to them. Remove also the common vertex to the two triangles, as well as the edges incident to it. For the remaining triangles, pick an arbitrary vertex from each of them and remove also all incident edges to these vertices. The remaining graph is a forest (and is in particular bipartite): every circuit not being a $T_i$ has lost at least one edge when the star centers have been removed, and every triangle $T_i$ has lost at least two edges. The number of vertices that have been removed is the number of colors minus $1$, although we have obtained a bipartite graph. This is again in contradiction with $\chi(\KG(G))=\cd_2(G)$.
\end{proof}

The starting claim in the proof, namely that in an optimal proper coloring with a minimum number of triangles, every circuit is either a triangle, or contains an edge from a star, provides an elementary proof of Dol'nikov's theorem when $\H$ is a graph $G$: removing one vertex per triangle and the center of each star leads necessarily to a forest, and the number of removed vertices is at most $\chi(\KG(G))$. It also shows that if $G$ is such that $\chi(\KG(G))=\cd_2(G)$, then we never have to remove more vertices to get a matching than what we have to remove to get a bipartite graph.

\subsection{Complexity questions}

As we already mentioned, we do not know the complexity status of deciding whether $\chi(\KG(\H))=\cd_2(\H)$, even if $\H$ is a graph. We nevertheless state and prove two related results. The proof of the first proposition is sketched in \cite{Me14}.

\begin{proposition}
Computing $\cd_2(\H)$ is $\NP$-hard, even when $\H$ is a graph.
\end{proposition}

\begin{proof}
Let $G=(V,E)$ be a non-bipartite graph. Consider $H$ the graph obtained by taking the join of two disjoint copies of $G$. Now, consider $H'$ an induced bipartite subgraph of $H$. If $H'$ has vertices from both copies of $G$, then it means that in each of these copies a vertex cover of $G$ has been removed. Otherwise, all vertices of one copy and at least $\cd_2(G)$ vertices from the other copy have been removed. Thus, we have $\cd_2(H)=\min(2(|V|-\alpha(G)),|V|+\cd_2(G))$, where $\alpha(G)$ is the independence number of $G$ (remember that $|V|-\alpha(G)$ is the vertex cover number of $G$).

We have
\begin{equation}\label{eq:cd}
|V|\leq\cd_2(G)+2\alpha(G). 
\end{equation} 
Indeed, let $X$ be a subset of vertices removed from $V$ such that $G[V\setminus X]$ is bipartite. Each part of $G[V\setminus X]$ is an independent set of $G$. Thus
$|V|=|X|+|V\setminus X|\leq |X|+2\alpha(G)$ and we get inequality~\eqref{eq:cd}. This inequality implies then that $2(|V|-\alpha(G))\leq|V|+\cd_2(G)$ and consequently
$\cd_2(H) = 2(|V|-\alpha(G))$. Since computing $\alpha(G)$ for a nonbipartite graph is $\NP$-hard, we get the result.
\end{proof}

\begin{proposition}\label{prop:chi}
Computing $\chi(\KG(\H))$ is $\NP$-hard, even when $\H$ is a graph.
\end{proposition}

\begin{proof}
Let $G$ be a triangle-free graph. A coloring of $\KG(G)$ is a partition of the edges of $G$ into stars. Thus $\chi(\KG(G))$ is the vertex cover number of $G$, which is $\NP$-hard to compute, even when $G$ is triangle-free~\cite{Po74}.
\end{proof}

\subsection*{Acknowledgments} We are grateful to G\'abor Simonyi for interesting discussions that were especially beneficial to Section~\ref{sec:cd}. We also thank the reviewers for their thorough reading of our work and for all their remarks. Part of this work was done when the first author was visiting the Universit\'e Paris Est. He would like to acknowledge professor Fr\'ed\'eric Meunier for his generous support and hospitality. Also, the research of Hossein Hajiabolhassan was in part supported by a grant from IPM (No. 94050128).

\bibliographystyle{elsarticle-harv}

\bibliography{Kneser}
\end{document}